\documentclass[reqno]{amsart}
\usepackage{mathrsfs}
\usepackage{amssymb, amsmath, amsthm}
\usepackage{latexsym}
\usepackage{yfonts}
\usepackage{mathrsfs}
\usepackage{natbib}
\usepackage{graphicx,color}
\usepackage[pdftex,plainpages=false,colorlinks,hyperindex,bookmarksopen,linkcolor=blue,citecolor=red,urlcolor=red]{hyperref}
\usepackage[final]{showkeys}
\usepackage{bm}
\usepackage{subfig}
\DeclareMathAlphabet{\mathpzc}{OT1}{pzc}{m}{it}

\bibpunct{[}{]}{;}{n}{,}{,}

\newtheorem{te}{Theorem}[section]

\theoremstyle{definition}

\theoremstyle{os}
\newtheorem{os}[te]{Remark}

\theoremstyle{pr}
\newtheorem{pr}[te]{Proposition}

\theoremstyle{lem}

\theoremstyle{co}
\newtheorem{co}[te]{Corollary}

\theoremstyle{conj}

\numberwithin{equation}{section}

\allowdisplaybreaks

\begin{document}

	\title{Even-order pseudoprocesses on a circle and related Poisson kernels}
	\author{Enzo Orsingher}
	\email{enzo.orsingher@uniroma1.it}
	\author{Bruno Toaldo}
	\email{bruno.toaldo@uniroma1.it}	
	\address{Department of Statistical Sciences, Sapienza University of Rome}
	\keywords{Pseudoprocesses, Poisson kernel, Von Mises circular law, Fractional equations, Stable processes, Mittag-Leffler functions, Higher-order heat-type equations, Fractional Laplacian}
	\date{\today}
	\dedicatory{}
	\subjclass[2000]{60G20, 60J25, 30C40}

\begin{abstract}
Pseudoprocesses, constructed by means of the solutions of higher-order heat-type equations have been developed by several authors and many related functionals have been analyzed by means of the Feynman-Kac functional or by means of the Spitzer identity. We here examine even-order pseudoprocesses wrapped up on circles and derive their explicit signed density measures. We observe that circular even-order pseudoprocesses differ substantially from pseudoprocesses on the line because - for $t> \bar{t} > 0$, where $\bar{t}$ is a suitable $n$-dependent time value - they become real random variables. By composing the circular pseudoprocesses with positively-skewed stable processes we arrive at genuine circular processes whose distribution, in the form of Poisson kernels, is obtained. The distribution of circular even-order pseudoprocesses is similar to the Von Mises (or Fisher) circular normal and therefore to the wrapped up law of Brownian motion. Time-fractional and space-fractional equations related to processes and pseudoprocesses on the unit radius circumference are introduced and analyzed.
\end{abstract}

\maketitle

\section{Introduction and preliminaries}

Pseudoprocesses are connected with the fundamental solution of heat-type equations of the form
\begin{equation}
\frac{\partial}{\partial t} u_n(x, t) \, = \, c_n \frac{\partial^n}{\partial x^n} u_n(x, t), \qquad x \in \mathbb{R}, \, t>0, \, n \in \mathbb{N},
\label{1.1}
\end{equation}
where
\begin{equation}
c_n \, = \, 
\begin{cases}
(-1)^{\frac{n}{2}+1}, \qquad & \textrm{for even values of } n \\
\pm 1,  & \textrm{for odd values of } n,
\end{cases}
\end{equation}
subject to the initial condition
\begin{equation}
u(x, 0) \, = \, \delta (x).
\end{equation}
For $n > 2$ the fundamental solutions to \eqref{1.1} are sign-varying. By means of a Wiener-type approach some authors (see for example Albeverio et al. \cite{smorodina}, Daletsky \cite{dale1}, Daletsky and Fomin \cite{dale2}, Krylov \cite{krylov}, Ladohin \cite{nonpos}) have constructed pseudoprocesses which we denote by $X(t)$, $t>0$ or $X_n (t)$, if we specify the order of the governing equation. In these papers the set of real functions $x:t \in \left[ 0, \infty \right) \rightarrow x(t)$ (sample paths) and the cylinders
\begin{equation}
C \, = \, \left\lbrace x (t) : a_j \leq x(t_j) \leq b_j, \, j = 1, \cdots, n \right\rbrace
\label{1.2}
\end{equation}
have been considered.
By using the solutions $u_n$ to \eqref{1.1} the measure of cylinders is given as
\begin{equation}
\mu_n \left( C \right) \, = \, \int_{a_1}^{b_1} dx_1 \, \cdots \, \int_{a_n}^{b_n} dx_n \, \prod_{j=1}^n \, u_n \left( x_j - x_{j-1}, t_j - t_{j-1} \right).
\label{1.3}
\end{equation}
In \eqref{1.3} we denote by $u_n$
\begin{equation}
u_n(x, t) \, = \, \frac{1}{2\pi} \int_{-\infty}^\infty d\xi \,  e^{-i\xi x} e^{c_n (-i\xi)^n t}.
\label{16}
\end{equation}
For $n=2k$ and $c_{2k} = (-1)^{k+1}$ the integral \eqref{16} always converge as it does for the odd-order case.
The measure \eqref{1.3} is extended to the field generated by cylinders \eqref{1.2} for fixed $t_1 < \cdots < t_j < \cdots < t_n$. The signed measure obtained in this way is Markovian in the sense that
\begin{equation}
\mu_{x_0} \left\lbrace X\left( t+T \right) \in \mathpzc{B} \big| \mathpzc{F}_T \right\rbrace \, = \, \mu_{X(T)} \left\lbrace X(t) \in \mathpzc{B} \right\rbrace,
\end{equation}
where $\mathpzc{F}_T$ is the field generated as
\begin{equation}
\mathpzc{F}_T \, = \, \sigma \left\lbrace X(t_1) \in \mathpzc{B}_1, \cdots, X (t_n) \in \mathpzc{B}_n \right\rbrace,
\end{equation}
where $0 \leq t_1 \leq \cdots \leq t_n = T$.
More information on properties of pseudoprocesses can be found in Cammarota and Lachal \cite{camma},  Lachal \cite{lachal2003}  and Nishioka \cite{nishio}. For pseudoprocesses with drift the reader can consult Lachal \cite{lachal2008}.

In this paper we consider pseudoprocesses on the ring $\mathpzc{R}$ of radius one, denoted by $\Theta (t)$, $t>0$, whose signed density measures are governed by
\begin{equation}
\begin{cases}
\frac{\partial }{\partial t} v_n (\theta, t) \, = \, c_n \frac{\partial^n}{\partial \theta^n} v_n (\theta, t), \qquad \theta \in [0,2\pi), t >0, n \geq 2, \\
v_n (\theta, 0) \, = \, \delta (\theta).
\end{cases}
\label{problbound}
\end{equation}
The signed measures of pseudoprocesses on the line $X(t)$, $t>0$, and those on the unit-radius ring, $\Theta (t)$, $t>0$, can be related by
\begin{equation}
\left\lbrace \Theta (t) \in d\theta \right\rbrace \, = \,  \bigcup_{m=-\infty}^\infty \left\lbrace X(t) \in d(\theta + 2m\pi) \right\rbrace, \qquad 0 \leq \theta < 2\pi.
\label{traiettoriepseudo}
\end{equation}
This means that the pseudoprocess $\Theta$ has sample paths which are obtained from those of $X$ by wrapping them up around the circumference $\mathpzc{R}$. Counterclockwise moving sample paths of $\Theta$ correspond to increasing sample paths of $X$. 

For $n=2$ we have in particular the circular Brownian motion studied by Hartman and Watson \cite{normal}, Roberts and Ursell \cite{ursel}, Stephens \cite{biometrika}.
The pseudoprocesses running on $\mathpzc{R}$ are called circular pseudoprocesses and are denoted either by $\Theta (t)$, $t>0$, or $\Theta_n (t)$ if we want to clarify the order of the equation governing their distribution.
We concentrate our attention on the even-order case because the odd-order wrapped-up pseudoprocesses pose qualitatively different problems of convergence of their Fourier expansion. In view of \eqref{traiettoriepseudo} we can write
\begin{equation}
v_{2n} (\theta, t) \, = \, \sum_{m=-\infty}^\infty u_{2n} (\theta + 2m\pi, t), \qquad 0 \leq \theta < 2\pi.
\label{wrappingdistr}
\end{equation}
Equation \eqref{wrappingdistr} shows that the solution to \eqref{problbound} can be obtained by wrapping up the solution to \eqref{1.1} which reads
\begin{equation}
u_{2n} (x, t) \, = \, \frac{1}{2\pi} \int_{-\infty}^\infty d\xi \, e^{-i\xi x} e^{-\xi^{2n}t}, \qquad x \in \mathbb{R}, t>0.
\label{pari}
\end{equation}
The function \eqref{pari} has been investigated in special cases by Hochberg \cite{hochdebbi}, Krylov \cite{krylov}, Nishioka \cite{nishio} and more in general by Lachal \cite{lachal2003,lachal2008}. The sign-varying structure of \eqref{pari} has been discovered in special cases by Bernstein \cite{berna}, L\'evy \cite{levy}, Polya \cite{polia}, as early as at the beginning of the Twentieth century and has been more recently studied also by Li and Wong \cite{liwong}.

The Fourier series of \eqref{wrappingdistr} has the remarkably simple form
\begin{equation}
v_{2n} (\theta, t) \, = \, \frac{1}{2\pi} + \frac{1}{\pi} \sum_{k=1}^\infty e^{-k^{2n}t} \cos k\theta, \qquad \theta \in \left[ 0,2\pi \right).
\label{1.13}
\end{equation}
For $n=1$ we obtain the Fourier series of the law of the circular Brownian motion (see \cite{normal}). The function
\begin{equation}
v_2 (\theta, t) \, = \, \frac{1}{2\pi} + \frac{1}{\pi} \sum_{k=1}^\infty e^{-k^2 t} \cos k\theta 
\label{browniandens}
\end{equation}
is similar to the Von Mises circular normal
\begin{equation}
\mathpzc{v}(\theta, k) \, = \, \frac{e^{k\cos \theta}}{2\pi I_0 (k)} \, = \, \frac{1}{2\pi} \left( 1+2\sum_{m=1}^\infty \frac{I_m(k)}{I_0(k)} \, \cos m \theta \right), \qquad \theta \in \left[ 0,2\pi \right),
\label{vonmis}
\end{equation}
where
\begin{equation}
I_m(x) \, = \, \sum_{j=0}^\infty \left( \frac{x}{2} \right)^{2j+m} \frac{1}{k! \, \Gamma \left( m+j+1 \right)}
\end{equation}
is the $m$-th order Bessel function.
The relationship between \eqref{browniandens} and \eqref{vonmis} is investigated in the paper by Hartman and Watson \cite{normal}. The Von Mises circular normal represents the hitting distribution of the circumference $\mathpzc{R}$ of a Brownian motion with drift starting from the center of $\mathpzc{R}$. The planar Brownian motion $\left( R(t), \, \Psi (t) \right)$, $t>0$, with drift $\bm{k} = \left( k_1, k_2 \right)$, $\left\| \bm{k} \right\| = k$, has transition function
\begin{equation}
\Pr \left\lbrace R(t) \in d\rho, \, \Psi (t) \in d\varphi \right\rbrace \, = \, \frac{\rho}{2\pi t} e^{-\frac{\rho^2}{2t}} e^{-\frac{k^2t}{2}} e^{\rho k \cos \varphi} d\rho d\varphi
\end{equation}
and marginal
\begin{equation}
\Pr \left\lbrace R(t) \in d\rho \right\rbrace \, = \, \frac{\rho}{t} e^{-\frac{\rho^2}{2t}} e^{-\frac{k^2t}{2}}   I_0 \left( \rho k \right) d\rho.
\end{equation}
Therefore
\begin{equation}
\Pr \left\lbrace  \Psi (t) \in d\varphi \big| R(t) \in d\rho \right\rbrace \, = \, \frac{e^{\rho k \cos \varphi}}{2\pi I_0 \left( \rho k \right)} \, d\varphi
\end{equation}
and for $\rho = 1$ coincides with \eqref{vonmis}.

The analysis of the pictures of $v_{2n} (\theta, t)$ for different values of $t$ and different values of the order $2n$, $n \in \mathbb{N}$, shows that the distributions \eqref{1.13} after a certain time become non-negative. This means that pseudoprocesses on the circle $\mathpzc{R}$ behave differently from their counterparts on the line and rapidly become genuine random variables. Furthermore we remark that in small initial intervals of time the circular pseudoprocesses have signed-valued distributions with a number of minima which rapidly unify into a single minimum (located at $\theta = \pi$) which for increasing $t$ upcrosses the zero level. This is due to the fact that in a small initial interval of time the effect produced by the central bell of the distribution has not yet spread on the whole ring $\mathpzc{R}$. 

The value of the absolute minimum of $v_{2n}(\theta, t)$ for $t > \bar{t}$ has the form
\begin{equation}
v_{2n} (\pi, t) \, = \, \frac{1}{2\pi} + \frac{1}{\pi} \sum_{k=1}^\infty (-1)^k e^{-k^{2n}t}, \qquad t > \bar{t}.
\label{1.16}
\end{equation}
The graph of functions $v_{2n}(\theta, t)$ slightly differ from that of the density of circular Brownian motion as shown in figures \ref{figura1} and \ref{figura2}. The term $k=1$ in \eqref{1.13} is the leading term of the series and the form of the distribution $v_{2n}(\theta, t)$ is very close to that of $\frac{1}{2\pi} + \frac{1}{\pi} e^{-t} \cos \theta$.

\begin{figure} [htp!]
	\centering
		\caption{The distributions of the fourth-order circular pseudoprocess for different values of $t$}
		\includegraphics[scale=0.23]{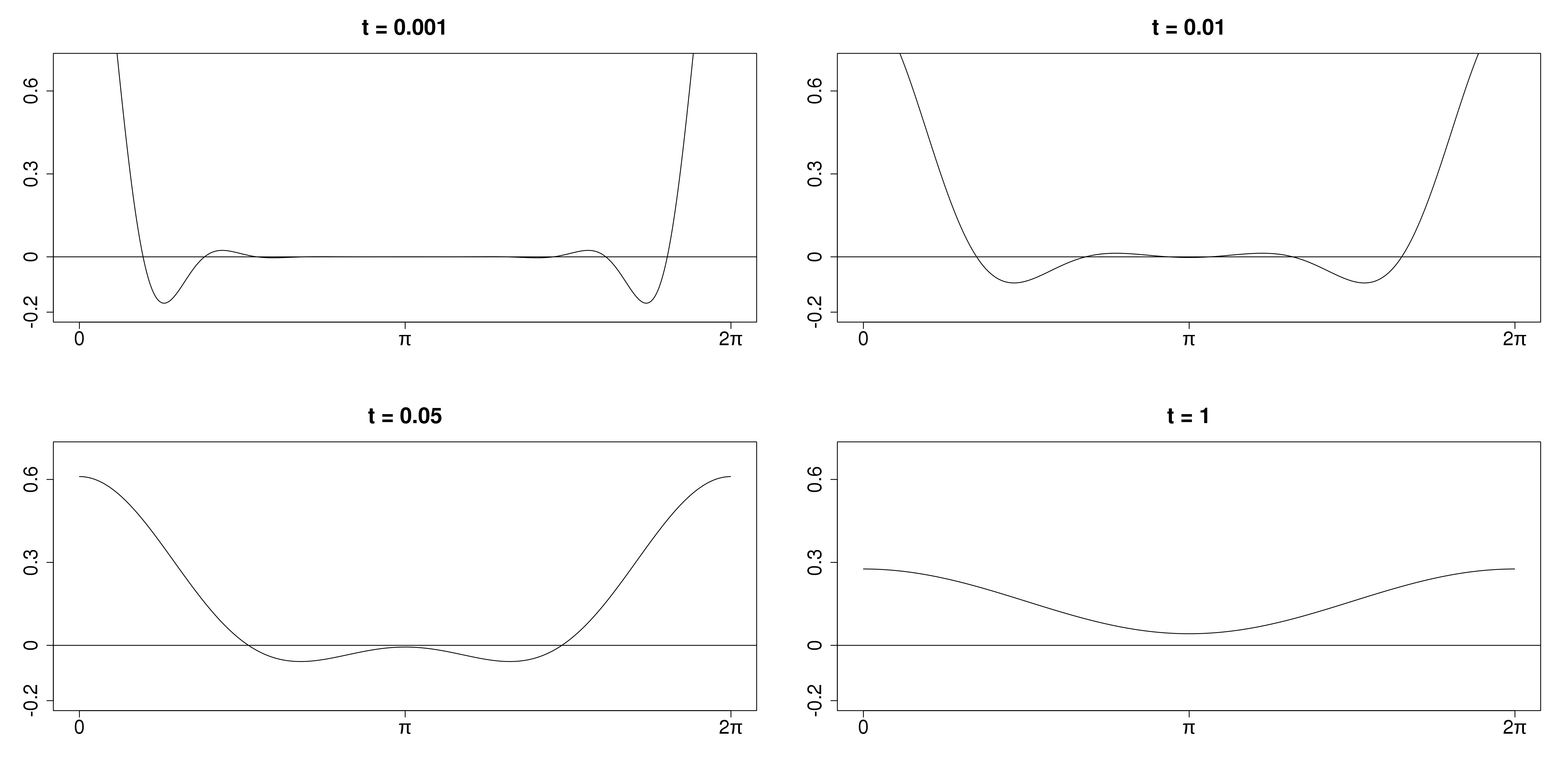} 
		\label{figura1}
\end{figure}
\begin{figure} [htp!]
	\centering
		\caption{The distributions (for $t=1$) of the circular pseudoprocesses of various order $2n$}
		\includegraphics[scale=0.23]{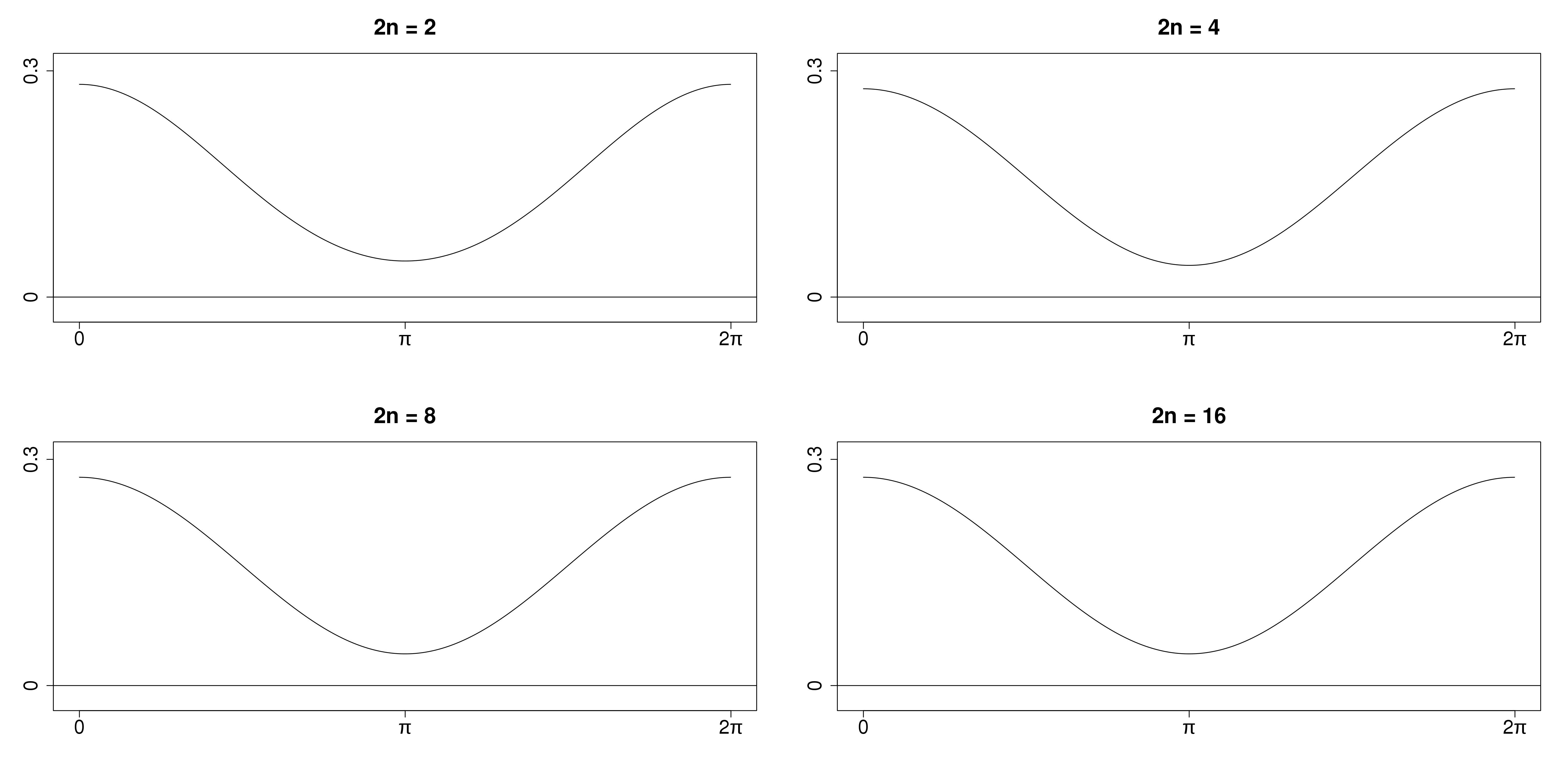} 
		\label{figura2}
\end{figure}

The odd-order case is much more complicated because the solutions to equation \eqref{1.1} are asymmetric (with asymmetry decreasing for increasing values of the order $n$). Some properties of solutions to odd-order heat-type equations can be found in \cite{lachal2003,lachal2008}. In the present paper the wrapped up solution to \eqref{1.1} gives the fundamental solution of \eqref{problbound} as
\begin{equation}
v_{2n+1} (\theta, t) \, = \, \sum_{k=-\infty}^\infty u_{2n+1} (\theta + 2k\pi, t), \qquad \theta \in \left[ 0,2\pi \right),
\end{equation}
whose Fourier series reads
\begin{align}
v_{2n+1} (\theta, t) \, = \, \frac{1}{2\pi} + \frac{1}{\pi} \sum_{k=1}^\infty \cos (k^{2n+1} t + k \theta).
\label{seriedisp}
\end{align}
We note that for $n=1$ the series \eqref{seriedisp} becomes a discrete version of the solution to \eqref{1.1} which reads
\begin{equation}
u_3 (x, t) \, = \, \frac{1}{\sqrt[3]{3t}} \, \textrm{Ai} \left( \frac{x}{\sqrt[3]{3t}} \right)
\end{equation}
where
\begin{equation}
\textrm{Ai} (x) \, = \, \frac{1}{\pi} \int_0^\infty \cos \left( \alpha x + \frac{\alpha^3}{3} \right) \, d\alpha
\end{equation}
is the Airy function. The probabilistic representations of solutions of higher-order heat-type equations show that for increasing values of $n$ the solutions $u_{2n+1}(x, t)$ and $u_{2n}(x, t)$ slightly differ. Therefore the corresponding circular version $v_{2n+1} (\theta, t)$ must have Fourier transform which converge since $v_{2n} (\theta, t)$ do.

We consider also the wrapped up stable processes $\mathfrak{S}^{2\beta} (t)$, $t>0$, and the related governing space-fractional equation.In particular we show that the law of $\mathfrak{S}^{2\beta} ( 2^{-\beta} t)$, $t>0$, is the fundamental solution of the space-fractional equations
\begin{align}
\begin{cases}
\frac{\partial}{\partial t} v_{2}^{\beta} (\theta, t) \, = \, - \left( - \frac{1}{2} \frac{\partial^2}{\partial \theta^2} \right)^\beta v_2^\beta (\theta, t), \qquad \theta \in \left[ 0,2\pi \right), \, t>0, \beta \in \left( 0,1 \right], \\
v_2^\beta (\theta, 0) \, = \, \delta (\theta),
\end{cases}
\label{problfraccalorebound}
\end{align}
and has Fourier expansion
\begin{equation}
v_2^\beta (\theta, t) \, = \, \frac{1}{2\pi} \left[ 1+2\sum_{m=1}^\infty e^{-\left( \frac{m^2}{2} \right)^\beta t}  \cos m \theta \right].
\label{calorebound}
\end{equation}
The fractional operator appearing in \eqref{problfraccalorebound} is the one-dimensional fractional Laplacian which can be defined by means of the Bochner representation (see, for example, Balakrishnan \cite{balakrishnan}, Bochner \cite{bochner})
\begin{equation}
-\left( -\frac{1}{2} \frac{\partial^2}{\partial \theta^2} \right)^\beta \, = \, \frac{\sin \pi \beta}{\pi} \int_0^\infty \left( \lambda + \left( -\frac{1}{2} \frac{\partial^2}{\partial \theta^2} \right) \right)^{-1} \lambda^\beta \, d\lambda, \qquad \beta \in (0,1).
\label{laplfraz}
\end{equation}
We show that formula \eqref{calorebound} coincides with the distribution of the subordinated Brownian motion on the circle, $\mathfrak{B} \left(  H^\beta (t) \right)$, $t>0$, where $H^\beta (t)$, $t>0$, is a stable subordinator of order $\beta \in \left( 0,1 \right]$ (see, for example, Baeumer and Meerschaert \cite{baem}). Furthermore we notice that
\begin{equation}
\mathfrak{B} \left( 2 \, H^\beta (t) \right) \, \stackrel{\textrm{law}}{=} \, \mathfrak{S}^{2\beta} (t), \qquad t>0,
\end{equation}
$\mathfrak{S}^{2\beta} (t)$, $t>0$, is a symmetric process on the ring $\mathpzc{R}$ with distribution which can be obtained by its symmetric stable counterpart on the line as
\begin{align}
p_{\mathfrak{S}^{2\beta}} (\theta, t) \, = \, & \sum_{m=\-\infty}^\infty \frac{1}{2\pi} \int_{-\infty}^\infty d\xi \, e^{-i \xi \left( \theta + 2m\pi \right)} e^{-t |\xi |^{2\beta}} \notag \\
 = \, & \frac{1}{2\pi} \left[ 1+2\sum_{k=1}^\infty e^{-k^{2\beta} t}  \cos k \theta \right].
\end{align}
For $\beta = \frac{1}{2}$ we extract from \eqref{calorebound} the Poisson kernel
\begin{equation}
v_{2}^1 (\theta, t) \, = \, \frac{1}{2\pi} \frac{1-e^{-t\sqrt{2}}}{1+e^{-t\sqrt{2}}-2e^{-\frac{t}{\sqrt{2}}} \cos \theta}.
\end{equation}

The composition of the circular pseudoprocesses $\Theta_n(t)$, $t>0$, with positively-skewed stable processes of order $\frac{1}{n}$, say $H^{\frac{1}{n}} (t)$, $t>0$, leads also to the Poisson kernel. In particular, we show that
\begin{equation}
\Pr \left\lbrace \Theta_{2n} \left( H^{\frac{1}{2n}} (t) \right) \in d\theta \right\rbrace \, = \, \frac{d\theta}{2\pi} \frac{1-e^{-2t}}{1+e^{-2t}-2e^{-t}\cos \theta}, \qquad \theta \in \left[ 0,2\pi  \right).
\label{kernelpari}
\end{equation}
In the odd-order case the result is different, depends on $n$ and has the following form for $\theta \in \left[ 0,2\pi \right)$
\begin{align}
&\Pr \left\lbrace \Theta_{2n+1} \left( H^{\frac{1}{2n+1}} (t) \right) \in d\theta \right\rbrace \notag \\
  = \, &  \frac{d\theta}{2\pi} \frac{1-e^{-2t\cos \frac{\pi}{2(2n+1)}}}{1+e^{-2t\cos \frac{\pi}{2(2n+1)}}-2e^{-t\cos\frac{\pi}{2(2n+1)}}\cos \left( \theta + t \sin \frac{\pi}{2(2n+1)} \right)}.
\label{kerneldisp}
\end{align}
The composition of pseudoprocesses with stable processes therefore produces genuine r.v.'s on the ring $\mathpzc{R}$ as it happens on the line (see Orsingher and D'Ovidio \cite{ecporsdov}). We note that the distribution of the composition in the even order case is independent from $n$ (formula \eqref{kernelpari}), while in the odd-order case the Poisson kernel obtained depends on $n$ and has a rather complicated structure. For $n \to \infty$ the kernel \eqref{kerneldisp} converges pointwise to \eqref{kernelpari} since the asymmetry of the fundamental solutions of \eqref{1.1} (as well as that of their wrapped up counterparts) decreases. 
The result \eqref{kernelpari} offers an interesting interpretation. The Poisson kernel \eqref{kernelpari} can be viewed as the probability that a planar Brownian motion starting from the point with polar coordinates $\left( e^{-t}, 0 \right)$ hits the circumference $\mathpzc{R}$ in the point $ \left( 1, \Theta \right)$ (see Fig. \ref{kernparifig}). Therefore this distribution coincides with the law of an even-order pseudoprocess running on the circumference and stopped at time $H^{\frac{1}{2n}} (t)$, $t>0$. This result is independent from $n$ and therefore is valid also for Brownian motion. A similar interpretation holds also for circular odd-order pseudoprocesses taken at the time $H^{\frac{1}{2n+1}} (t)$, $t>0$, but starting from the point with polar coordinates $\left( e^{-a_n t}, b_n t \right)$, where $a_n = \cos \pi / (2(2n+1))$ and $b_n = \sin \pi / (2(2n+1))$.

\section{Pseudoprocesses on a ring}
In this section we consider pseudoprocesses $\Theta (t)$, $t>0$, on the unit-radius circumference $\mathpzc{R}$, whose density function $v_n (\theta, t)$, $\theta \in \left[ 0,2\pi \right)$, $t>0$, is governed by the higher order heat-type equation
\begin{equation}
\begin{cases}
\frac{\partial}{\partial t} v_{n}(\theta, t) \, = \, c_n \frac{\partial^n}{\partial \theta^n} v_n(\theta, t), \qquad  \theta \in \left[ 0,2\pi \right), \, t >0, \, n \geq 2, \\
v_{n} (\theta, 0) \, = \, \delta (\theta).
\end{cases}
\label{2.1}
\end{equation}
The pseudoprocesses $\Theta_n$ have sample paths obtained by wrapping up the trajectories of pseudoprocesses on the line $\mathbb{R}$. Increasing sample paths on $\mathpzc{R}$ correspond to counterclockwise moving motions on the ring $\mathpzc{R}$.
The structure of sample paths of pseudoprocesses has not been investigated in detail although some results by Lachal (Theorem 5.2, \cite{lachal2008}) show that there is a sort of "slight" discontinuity in their behaviour (this is confirmed by Hochberg \cite{hochdebbi}) and the fact that the reflection principle fails (Beghin et al. \cite{ragoz}, Lachal \cite{lachal2003}).

 It must be considered that the wrapping up of the sample paths and of the corresponding density measures produces in the long run genuine random variables (with non-negative measure densities in the case $n$ is even). Our first result concerns the distribution of $\Theta_n(t)$, $t>0$.
\begin{te}
\label{teoremaleggipseudo}
The solutions to the even-order heat-type equations \eqref{2.1} reads
\begin{equation}
v_{2n} (\theta, t) \, = \, \frac{1}{2\pi} + \frac{1}{\pi} \sum_{k=1}^\infty e^{-k^{2n}t} \cos k\theta, \qquad \textrm{for } c_{2n}  = (-1)^{n+1}, \,  n \geq 1, \\
\label{leggipseudoteo}
\end{equation}
\end{te}
\begin{proof}
We can obtain the result \eqref{leggipseudoteo} in two different ways. We start by considering the even-order case where the wrapping up of the solutions to \eqref{1.1} which leads to
\begin{align}
v_{2n} (\theta, t) \, = \, \sum_{m=-\infty}^\infty u_{2n} (\theta + 2m\pi, t) \, = \, \sum_{m=-\infty}^\infty \int_{-\infty}^\infty d\xi \, e^{-i \left( \theta + 2m\pi \right) \xi} e^{-\xi^{2n}t}.
\end{align}
The Fourier series expansion of the symmetric function $v_{2n}(\theta, t)$ has coefficients
\begin{align}
a_k \, = \, & \frac{1}{\pi} \int_0^{2\pi} d\theta \, \cos k\theta \left[ \sum_{m=-\infty}^\infty u_{2n} (\theta + 2m\pi, t) \right] \notag \\
= \, & 2 \sum_{m=-\infty}^\infty \int_{m}^{m+1} dy  \; u_{2n} (2\pi y, t) \, \cos 2\pi k y \notag \\
= \, & \frac{1}{\pi} \int_{-\infty}^\infty dz \, \cos kz \, \left( \frac{1}{2\pi} \int_{-\infty}^\infty d\xi \, e^{-i\xi z} e^{-\xi^{2n}t} \right) \notag \\
= \, & \frac{1}{2\pi} \int_{-\infty}^\infty d\xi \, e^{-\xi^{2n}t} \left[ \frac{1}{2\pi} \int_{-\infty}^\infty dz \left( e^{iz (k-\xi )} + e^{-iz (k+\xi)} \right)  \right] \notag \\
= \, & \frac{1}{2\pi} \int_{-\infty}^\infty d\xi \, e^{-\xi^{2n}t} \left[ \delta (\xi -k) + \delta (\xi + k) \right] \, = \, \frac{e^{-k^{2n}t}}{\pi}.
\end{align}
An alternative derivation of $v_{2n} (\theta, t)$ is based on the method of separation of variables. Thus under the assumption that $v_{2n} (\theta, t) = T(t) \psi (\theta)$ we get
\begin{equation}
\frac{T^{(1)} (t)}{T(t)} \, = \, \frac{\psi^{(2n)} (\theta)}{\psi (\theta)} (-1)^{n+1} \, = \, -\beta^{2n}.
\end{equation}
In order to have periodic solutions we must take integer values of $\beta$ and thus the general solution to \eqref{2.1} becomes
\begin{equation}
v_{2n} (\theta, t) \, = \, \sum_{k=-\infty}^\infty A_k \, e^{-k^{2n}t} \cos k\theta \, = \, A_0 + 2\sum_{k=1}^\infty A_k \, e^{-k^{2n}t} \cos k\theta.
\end{equation}
The initial condition
\begin{equation}
v_{2n} (\theta, 0) \, = \, \delta (\theta) \, = \, \frac{1}{2\pi} + \frac{1}{\pi} \sum_{k=1}^\infty \cos k\theta
\end{equation}
implies that $A_k = \frac{1}{2\pi}$, which confirms the result.
\qed
\end{proof}
\begin{pr}
We are able to give a third derivation of \eqref{leggipseudoteo} by resorting to the probabilistic representation of fundamental solutions to even-order heat-type equations of Orsingher and D'Ovidio \cite{ecporsdov} which reads
\begin{align}
u_{2n} (x, t) \, = \, \frac{1}{\pi x} \mathbb{E} \left\lbrace \sin \left( x G^{2n} \left( \frac{1}{t} \right) \right) \right\rbrace
\label{2288}
\end{align}
for $c_n = (-1)^{n+1}$, $n \geq 1$. In \eqref{2288} $G^\gamma (t^{-1})$ is a generalized gamma r.v. with density
\begin{equation}
g^\gamma (x, t) \, = \, \gamma \frac{x^{\gamma-1}}{t} e^{-\frac{x^\gamma}{t}}, \qquad x>0, \, t>0, \, \gamma >0.
\label{densitagamma}
\end{equation}
\end{pr}
\begin{proof}
We start the proof by wrapping-up the representation \eqref{2288} as follows
\begin{align}
v_{2n} (\theta, t) \, = \, \sum_{m=-\infty}^\infty \frac{1}{\pi (\theta + 2m\pi)} \mathbb{E} \left\lbrace \sin \left( (\theta + 2m\pi) G^{2n} \left( \frac{1}{t} \right) \right) \right\rbrace.
\label{2299}
\end{align}
Now we evaluate the Fourier coefficients of \eqref{2299} as
\begin{align}
a_k \, = \, & \frac{1}{\pi^2} \sum_{m=-\infty}^\infty \int_0^{2\pi} \frac{d\theta}{\theta + 2m\pi} \mathbb{E} \left\lbrace \sin (\theta + 2m\pi ) G^{2n} \left( \frac{1}{t} \right) \right\rbrace \, \cos k\theta \notag \\
= \, & \frac{1}{\pi^2} \mathbb{E} \left\lbrace \int_{-\infty}^\infty dz \frac{\cos kz}{z} \sin \left( z \, G^{2n} \left( \frac{1}{t} \right) \right) \right\rbrace \notag \\
= \, & \frac{1}{\pi^2} \mathbb{E} \left\lbrace \int_0^\infty dz \, \left[ \frac{\sin \left( kz + z \, G^{2n} \left( \frac{1}{t} \right) \right)}{z} + \frac{\sin \left( z \, G^{2n} \left( \frac{1}{t} \right) -kz \right)}{z} \right] \right\rbrace \notag \\
= \, &\frac{1}{\pi} \mathbb{E} \left\lbrace \mathbb{I}_{ \left[-G^{2n} \left( \frac{1}{t} \right) <  k < G^{2n} \left( \frac{1}{t} \right) \right]} \right\rbrace \, = \, \frac{1}{\pi} \Pr \left\lbrace G^{2n} \left( \frac{1}{t} \right) > k \right\rbrace \, = \, \frac{1}{\pi} e^{-k^{2n}t},
\label{215}
\end{align}
where we used the fact that
\begin{equation}
\int_0^\infty dx \, \frac{\sin \alpha x}{x} \, = \, 
\begin{cases}
\frac{\pi}{2}, \qquad & \textrm{if } \, \alpha > 0, \\
-\frac{\pi}{2}, & \textrm{if } \, \alpha < 0.
\end{cases}
\end{equation}
\qed
\end{proof}
The calculation \eqref{215} shows that the density of even-order circular pseudoprocesses can be viewed as the superposition of sinusoidal waves whose amplitude corresponds to the tails of a Weibull distribution.

For the odd-order pseudoprocess we proceed formally as in the even-order case and the Fourier cofficients of
\begin{equation}
v_{2n+1} (\theta, t) \, = \, \sum_{m=-\infty}^\infty u_{2n+1} (\theta + 2m\pi, t)
\end{equation}
become
\begin{align}
a_k \, = \, & \frac{1}{2\pi} \int_{-\infty}^\infty d\xi \, e^{(-1)^n(-i\xi)^{2n+1}t} \left[ \frac{1}{2\pi} \int_{-\infty}^\infty dz \, \left( e^{i(k-\xi)z} + e^{-i (k+\xi)z}  \right) \right] \notag \\
= \, & \frac{1}{2\pi} \int_{-\infty}^\infty d\xi \, e^{-i\xi^{2n+1}t} \left[ \delta (\xi -k) + \delta (\xi + k) \right] \notag \\
= \, & \frac{1}{2\pi} \left[ e^{-ik^{2n+1}} + e^{ik^{2n+1}} \right] \, = \, \frac{1}{\pi} \cos k^{2n+1} t.
\end{align}
In a similar way we have that
\begin{align}
b_k \, = \, & \frac{1}{2\pi} \int_{-\infty}^\infty d\xi \, e^{(-1)^n (-i\xi)^{2n+1}t} \left[ \frac{1}{2\pi i} \int_{-\infty}^\infty dz \, \left( e^{i(k-\xi)z} - e^{-i(k+\xi)z} \right) \right] \notag \\
= \, & \frac{1}{2\pi i} \int_{-\infty}^\infty d\xi \, e^{-i\xi^{2n+1}t} \left[ \delta (\xi - k) - \delta (\xi + k) \right] \, = \, -\frac{1}{\pi} \sin k^{2n+1}t,
\end{align}
and thus the expression of the distribution of the odd-order pseudoprocess on the circle $v_{2n+1} (\theta, t)$ becomes
\begin{align}
v_{2n+1}(\theta, t) \, = \, \frac{1}{2\pi} + \frac{1}{\pi} \sum_{k=1}^\infty \cos \left( k^{2n+1} t + k \theta \right).
\label{incriminata}
\end{align}
For $n=1$ the series \eqref{incriminata} is similar to the integral representation of the Airy function
\begin{align}
Ai(x) \, = \, \frac{1}{\pi} \int_0^\infty \cos \left( \alpha x + \frac{\alpha^3}{3} \right) \, d\alpha.
\end{align}
We are not able to give a rigorous proof of the convergence of the series \eqref{incriminata} but we are able to give an alternative derivation as follows. In particular we can obtain the expansion \eqref{incriminata} for circular odd-order pseudoprocesses by resorting again to the probabilistic representation of the law of pseudoprocesses of \cite{ecporsdov} which reads
\begin{equation}
u_{2n+1} (x, t) \, = \, \frac{1}{\pi x} \mathbb{E} \left\lbrace e^{-b_n x \, G^{2n+1} \left( \frac{1}{t} \right)} \sin \left( a_n x \, G^{2n+1} \left( \frac{1}{t} \right) \right) \right\rbrace,
\label{rapprprob}
\end{equation}
where $c_{2n+1} = (-1)^n$,  $n \geq 1$, $G^\gamma \left( t^{-1} \right)$ is a generalized gamma r.v. with density \eqref{densitagamma} and
\begin{equation}
a_n \, = \, \cos \frac{\pi}{2(2n+1)}, \qquad b_n \, = \, \sin \frac{\pi}{2(2n+1)}.
\end{equation}
By wrapping-up \eqref{rapprprob} we obtain
\begin{align}
&v_{2n+1} (\theta, t) =  \notag \\
= \, & \sum_{m=-\infty}^\infty \frac{1}{\pi (\theta + 2m\pi)} \mathbb{E} \left\lbrace e^{-b_n (\theta + 2m\pi) G^{2n+1} \left( \frac{1}{t} \right)} \sin \left( a_n (\theta + 2m\pi) G^{2n+1} \left( \frac{1}{t} \right) \right) \right\rbrace.
\label{214}
\end{align}
We prove that the Fourier series expansion of \eqref{214} coincides with \eqref{incriminata}.We need both the sine and cosine coefficients of the Fourier expansion because the signed laws are asymmetric. The Fourier coefficients become
\begin{equation}
\begin{cases}
a_k \, = \, \frac{1}{\pi} \cos k^{2n+1} t, \\
b_k \, = \, - \frac{1}{\pi} \sin k^{2n+1} t.
\end{cases}
\label{coeffcasodispari}
\end{equation}
We give with some details the evaluation of \eqref{coeffcasodispari}
\begin{align}
a_k = & \frac{1}{\pi^2} \sum_{m=-\infty}^\infty \int_0^{2\pi} d\theta \frac{\cos k\theta}{\theta + 2m\pi} \mathbb{E} \left\lbrace e^{-b_n (\theta + 2m\pi) G^{2n+1} \left( \frac{1}{t} \right)} \sin \left( a_n (\theta + 2m\pi) G^{2n+1} \left( \frac{1}{t} \right) \right)  \right\rbrace \notag \\
= \, & \frac{1}{\pi^2} \mathbb{E} \left\lbrace \int_{-\infty}^\infty dz \frac{\cos kz}{z} e^{-b_n z G^{2n+1} \left( \frac{1}{t} \right)} \sin \left(a_n zG^{2n+1} \left( \frac{1}{t} \right) \right) \right\rbrace \notag \\
= \, &\frac{1}{2\pi^2} \mathbb{E} \left\lbrace \int_{-\infty}^\infty dz \frac{ \sin \left( z \left( a_n G^{2n+1} \left( \frac{1}{t} \right) + k \right) \right) + \sin \left( z \left( a_n G^{2n+1} \left( \frac{1}{t} \right) -k \right) \right)}{z}  e^{-b_n z G^{2n+1} \left( \frac{1}{t} \right)} \right\rbrace \notag \\
= \, & \frac{1}{2i2\pi^2} \mathbb{E} \bigg\lbrace \int_{-\infty}^\infty \frac{dz}{z} \left[ e^{iz \left( a_n G^{2n+1} \left( \frac{1}{t} \right) +k \right)-b_n z G^{2n+1} \left( \frac{1}{t} \right)} - e^{-iz \left( a_n G^{2n+1} \left( \frac{1}{t} \right) + k \right)-b_n z G^{2n+1} \left( \frac{1}{t} \right)}  \right. \notag \\
& \left.  +e^{iz \left( a_n G^{2n+1} \left( \frac{1}{t} \right) -k \right)-b_n z G^{2n+1} \left( \frac{1}{t} \right)} - e^{-iz \left( a_n G^{2n+1} \left( \frac{1}{t} \right) - k \right)-b_n z G^{2n+1} \left( \frac{1}{t} \right)} \right]  \bigg\rbrace.
\label{aconcappa}
\end{align}
By considering the following integral representation of the Heaviside function
\begin{equation}
\mathcal{H}_y(x) \, = \, - \frac{1}{2\pi} \int_{\mathbb{R}} dw \, e^{-iwx} \frac{e^{iyw}}{iw} \, = \, \int_{\mathbb{R}} dw \, e^{iwx} \frac{e^{-iyw}}{iw}
\end{equation}
the coefficients $a_k$ in \eqref{aconcappa} become
\begin{align}
a_k \, = \, & \frac{(2n+1)t}{2\pi} \int_0^\infty dw \, w^{2n} e^{-t w^{2n+1} } \left[ \mathcal{H}_k (w (a_n-ib_n))  - \mathcal{H}_k (-w(a_n+ib_n)) \right. \notag \\
& \left. + \mathcal{H}_k (w(a_n+ib_n)) - \mathcal{H}_k (-w (a_n - i b_n)) \right] \notag \\
= \, & \frac{i(2n+1)t}{2\pi} \int_0^\infty dw \, w^{2n} e^{-iw^{2n+1}t} \mathcal{H}_k(w) + \frac{i(2n+1)t}{2\pi} \int_0^\infty dw \, w^{2n} e^{iw^{2n+1}t} \mathcal{H}_k (-w) \notag \\
& - \left[ \frac{i(2n+1)t}{2\pi} \int_0^\infty dw \, w^{2n} e^{iw^{2n+1}t} \mathcal{H}_k(w) + \frac{i(2n+1)t}{2\pi} \int_0^\infty dw \, w^{2n} e^{-iw^{2n+1}t} \mathcal{H}_k (-w) \right] \notag \\
= \, & \frac{i(2n+1)t}{2\pi} \left( \int_{-\infty}^\infty dw \, w^{2n} e^{-iw^{2n+1}t} \mathcal{H}_k (w) - \int_{-\infty}^\infty dw \, w^{2n} e^{iw^{2n+1}t} \mathcal{H}_k(w) \right) \notag \\
= \, & \frac{1}{2\pi} \left( e^{ik^{2n+1}t} + e^{-ik^{2n+1}t}  \right) \, = \, \frac{1}{\pi} \cos k^{2n+1}t.
\end{align}
In order to justify the last step we can either take the Laplace transform with respect to $t$ (see for example Orsingher and D'Ovidio \cite{ecporsdov}) or we can apply the following trick
\begin{equation}
a_k \, = \, \lim_{\zeta \to 0} \frac{i(2n+1)t}{2\pi} \left( \int_{k}^\infty dw \, e^{-\zeta w^{2n+1}} w^{2n} e^{-iw^{2n+1}t}  - \int_{k}^\infty dw \, e^{-\zeta w^{2n+1}} w^{2n} e^{iw^{2n+1}t}  \right).
\end{equation}
The coefficients $b_k$ \eqref{coeffcasodispari} can be obtained by performing similar calculation.

\subsection{Circular Brownian motion}
The circular Brownian motion $\mathfrak{B}(t)$, $t>0$, has been analyzed by Roberts and Ursell \cite{ursel}, Stephens \cite{biometrika} and also by Hartman and Watson \cite{normal}. In a certain sense it can be viewed as a special case of symmetric pseudoprocesses on the ring $\mathpzc{R}$. The distribution of $\mathfrak{B}(t)$, $t>0$, has Fourier representation
\begin{equation}
p_{\mathfrak{B}} (\theta, t) \, = \, \frac{1}{2\pi} \left( 1+2 \sum_{k=1}^\infty e^{-\frac{k^2t}{2}} \cos k\theta \right), \qquad \theta \in \left[ 0, 2\pi \right),
\label{leggebring}
\end{equation}
and can be also regarded as the wrapped up distribution of the standard Brownian motion
\begin{equation}
p_{\mathfrak{B}} (\theta, t) \, = \, \frac{1}{\sqrt{2\pi t}} \sum_{m=-\infty}^\infty e^{-\frac{(\theta + 2m\pi)^2}{2t}}.
\end{equation}
Formula \eqref{leggebring} corresponds to $n=1$ of \eqref{leggipseudoteo} for the even-order case with a suitable adjustement of the time scale. The law \eqref{leggebring} can be obtained directly by solving the Cauchy problem
\begin{equation}
\begin{cases}
\frac{\partial}{\partial t} p_{\mathfrak{B}} (\theta, t) \, = \, \frac{1}{2} \frac{\partial^2}{\partial \theta^2} p_{\mathfrak{B}} (\theta, t), \qquad \theta \in \left[ 0,2\pi \right), \, t>0, \\
p_{\mathfrak{B}} (\theta, 0) \, = \, \delta (\theta).
\end{cases}
\end{equation}
or as the limit of a circular random walk as in \cite{biometrika}. The distribution of the circular Brownian motion is depicted in Figure \ref{figura2} and looks like the Von Mises circular normal (this is the inspiring idea of the paper by Hartman and Watson \cite{normal} in which the connection between the two distributions is investigated). For $t \to \infty$ the distribution of $\mathfrak{B}(t)$, $t>0$, tends to the uniform law.

We note that
\begin{equation}
\Pr \left\lbrace -\frac{\pi}{2} < \mathfrak{B}(t) < \frac{\pi}{2} \right\rbrace \, = \, \frac{1}{2} + \frac{2}{\pi} \sum_{k=0}^\infty (-1)^k \, \frac{e^{-\frac{(2k+1)^2t}{2}}}{2k+1}
\end{equation}
and therefore
\begin{equation}
\Pr \left\lbrace -\frac{\pi}{2} < \mathfrak{B}(t) < \frac{\pi}{2} \right\rbrace \, \leq \, \frac{1}{2} + \frac{2}{\pi} e^{-\frac{t}{2}}, \qquad \textrm{valid for } \, t > -2\log \frac{\pi}{4} \, \approx \, 0.209.
\end{equation}

The relationship between circular Brownian motion $\mathfrak{B}(t)$, $t>0$, and Brownian motion on the line $B(t)$, $t>0$, 
\begin{equation}
\left\lbrace \mathfrak{B}(t) \in d\theta \right\rbrace \, = \, \bigcup_{m=-\infty}^\infty \left\lbrace B(t) \in d(\theta + 2m\pi) \right\rbrace, \qquad \theta \in \left[ 0,2\pi \right),
\end{equation}
permits us to derive the distribution of
\begin{equation}
\max_{0 \leq s \leq t} \left| \mathfrak{B} (t) \right|, \qquad t>0,
\label{maxdist}
\end{equation}
that is the distribution of the maximal distance reached by the circular Brownian motion from the starting point.
Of course the sample paths overcoming the angular distance $\pi$ at least once are assigned $\pi$ as maximal distance which therefore has a positive probability (converging to 1 as time tends to infinity).
\begin{pr}
For the maximal distance \eqref{maxdist} we have that
\begin{align}
\Pr \left\lbrace \max_{0 \leq s \leq t} \left| \mathfrak{B}(s) \right| < \theta \right\rbrace \, = \, & \int_{-\theta}^\theta \Pr \left\lbrace -\theta < \min_{0\leq s \leq t} B(s) < \max_{0\leq s \leq t} B(s) < \theta \right\rbrace \notag \\
= \, & \int_{-\theta}^\theta dy \, \left( \sum_{m=-\infty}^\infty \frac{e^{-\frac{(y-4m\theta)^2}{2t}}}{\sqrt{2\pi t}} - \sum_{m=-\infty}^\infty \frac{e^{-\frac{(-y+2\theta (2m-1))^2}{2t}}}{\sqrt{2\pi t}} \right) \notag \\
= \, &\sum_{r=-	\infty}^\infty (-1)^r \int_{-\frac{(1+2r)\theta}{\sqrt{t}}}^{\frac{(1-2r)\theta}{\sqrt{t}}} dw \, \frac{e^{-\frac{w^2}{2}}}{\sqrt{2\pi}}.
\end{align}
The related first passage time of circular Brownian motion has density which has the following form
\begin{align}
\Pr \left\lbrace \mathcal{T}_\theta \in dt \right\rbrace \, = \, & - \frac{d}{dt} \Pr \left\lbrace \max_{0\leq s \leq t} \left| \mathfrak{B} (s) \right| < \theta \right\rbrace dt \notag \\
= \, & \sum_{r=-\infty}^\infty \left[ \frac{(-1)^r \, e^{-\frac{(1-2r)^2\theta^2}{2t}}}{2\sqrt{2\pi t^3}} \theta (1-2r) + \frac{(-1)^r \, e^{-\frac{(1+2r)^2\theta^2}{2t}}}{2\sqrt{2\pi t^3}} \theta (1+2r) \right] \notag \\
= \, & \left( \frac{\theta \, e^{-\frac{\theta^2}{2t}}}{\sqrt{2\pi t^3}} \right) \sum_{r=-\infty}^\infty (-1)^r e^{-\frac{2r^2\theta^2}{r}} \left( \cosh \frac{2r\theta^2}{t} - 2r \sinh \frac{2r \theta^2}{t} \right)
\end{align}
\end{pr}
Curiously enough the factor $\theta e^{-\frac{\theta^2}{2t}} / \sqrt{2\pi t^3}$ coincides with the first passage time through $\theta$ of a Brownian motion on the line.

\section{Fractional equations on the ring $\mathpzc{R}$ and the related processes}
In this section we consider various types of processes on the unit radius circumference $\mathpzc{R}$.
\subsection{Higher-order time-fractional equations}
\label{timefraceq}
We start by analyzing the processes related to the solutions of time-fractional higher-order heat-type equations.
We consider the time-changed pseudoprocesses $\Theta_{2n} \left( L^\nu (t) \right)$, $t>0$, where
\begin{equation}
L^\nu (t) \, = \, \inf \left\lbrace s > 0: H^\nu (s) \geq t \right\rbrace
\label{inversosub}
\end{equation}
and where $H^\nu (t)$, $t>0$, is a positively skewed stable process of order $\nu \in \left( 0,1 \right]$. We notice that the Laplace transform of the distribution $l_\nu (x, t)$ of \eqref{inversosub} reads (see for example Orsingher and Toaldo \cite{toaldo})
\begin{equation}
\int_0^\infty dx \, e^{-\gamma x} l_\nu (x, t)  \, = \, E_{\nu,1} \left( -\gamma t^\nu \right)
\label{laplinverso}
\end{equation}
where
\begin{equation}
E_{\nu, 1} (x) \, = \, \sum_{j=0}^\infty \frac{x^j}{\Gamma (\nu j + 1)}, \qquad x \in \mathbb{R}, \nu >0,
\end{equation}
is the Mittag-Leffler function. For pseudoprocesses related to time-fractional equations we have the next theorem.
\begin{te}
\label{teorematimefrac}
The solution to the problem, for $\nu \in \left( 0,1 \right]$, $n \in \mathbb{N}$,
\begin{equation}
\begin{cases}
\frac{\partial^\nu}{\partial t^\nu} v^\nu_{2n} (\theta, t) \, = \, - \left( - \frac{\partial^{2}}{\partial \theta^{2}} \right)^n v^\nu_{2n} (\theta, t), \qquad \theta \in \left[ 0,2\pi \right), \, t>0, \\
v^\nu_{2n} (\theta, 0) \, = \, \delta (\theta),
\end{cases}
\label{timefractionalpseudo}
\end{equation}
is the univariate (signed) distribution of $\Theta_{2n} \left( L^\nu (t) \right)$, $t>0$, which reads
\begin{equation}
v^\nu_{2n} (\theta, t) \, = \, \frac{1}{2\pi} \left( 1+2\sum_{k=1}^\infty E_{\nu, 1} \left( -k^{2n} t^\nu \right) \cos k \theta \right).
\label{3.5}
\end{equation}
The time-fractional derivative in \eqref{timefractionalpseudo} must be understood in the Caputo sense, that is
\begin{equation}
\frac{\partial^\nu}{\partial t^\nu} v_{2n}^\nu (\theta, t) \, = \, \frac{1}{\Gamma (1-\nu)} \int_0^t \frac{\frac{\partial}{\partial s} v_{2n}^\nu (\theta, s)}{(t-s)^\nu} \, ds, \qquad 0 < \nu < 1.
\end{equation}
\end{te}
\begin{proof}
The law of $\Theta_{2n} \left( L^\nu (t) \right)$, $t>0$, is given by
\begin{align}
v^\nu_{2n} (\theta, t) \, = \, & \frac{1}{2\pi} \int_0^\infty ds \, \left( 1+2\sum_{k=1}^\infty e^{-k^{2n} s} \cos k \theta \right) l_\nu (s, t) \notag \\
= \, & \frac{1}{2\pi} \left( 1+2\sum_{k=1}^\infty E_{\nu, 1} \left( -k^{2n} t^\nu \right) \cos k \theta \right).
\label{leggetimefrac}
\end{align}
Since, $\forall k \geq 1$, we have that
\begin{align}
\frac{\partial^\nu}{\partial t^\nu} E_{\nu, 1} \left( - k^{2n} t^\nu \right)  \cos k\theta \,  = & \,  - k^{2n} E_{\nu, 1} \left( - k^{2n}t^\nu \right)  \cos k\theta \notag \\
  = \, &  (-1)^{n+1} \frac{\partial^{2n}}{\partial \theta^{2n}} E_{\nu, 1} \left( - k^{2n} t^\nu \right) \cos k \theta,
\end{align}
and therefore we conclude that \eqref{leggetimefrac} satisfies the fractional equation \eqref{timefractionalpseudo}. 
\end{proof}
\begin{os}
For $n=1$, formula \eqref{leggetimefrac} becomes the distribution of subordinated Brownian motion $\mathfrak{B} \left( L^\nu (t) \right)$, $t>0$. For $\nu = 1$ we retrieve from \eqref{3.5} the solutions \eqref{leggipseudoteo} of the even-order heat-type equations on $\mathpzc{R}$.
\end{os}

\subsection{Space-fractional equations and wrapped up stable processes}
\label{spacefraceq}
The following Theorem represents the counterpart on $\mathpzc{R}$ of the Riesz statement on the relationship between space-fractional equations and symmetric stable laws (for the non-symmetric case see the paper by Feller \cite{feller52}).
\begin{te}
\label{teoremaspacefrac}
The law of the process $\mathfrak{B} \left( H^\beta (t) \right)$, $t>0$, is given by
\begin{align}
p_{\mathfrak{B}}^\beta (\theta, t) \, = \, \frac{1}{2\pi} \left[ 1 + 2 \sum_{k=1}^\infty e^{- \left( \frac{k^2}{2} \right)^\beta t} \cos k \theta \right]
\label{leggebrownspacefrac}
\end{align}
and solves the space-fractional equation, for $\beta \in \left( 0,1 \right]$,
\begin{align}
\begin{cases}
\frac{\partial}{\partial t} p_{\mathfrak{B}}^\beta (\theta, t) \, = \, - \left( - \frac{1}{2} \frac{\partial^2}{\partial \theta^2} \right)^\beta p_{\mathfrak{B}}^\beta (\theta, t), \qquad \theta \in \left[ 0,2\pi \right), \, t >0 \\
p_{\mathfrak{B}}^\beta (\theta, 0) \, = \, \delta (\theta).
\end{cases}
\label{brownspacefrac}
\end{align}
The fractional one-dimensional Laplacian in \eqref{brownspacefrac} is defined in \eqref{laplfraz} and $H^\beta (t)$, $t>0$, is a stable subordinator of order $\beta \in \left( 0,1 \right]$.
\end{te}
\begin{proof}
The law of $\mathfrak{B} \left( H^\beta (t) \right)$, $t>0$, is given by
\begin{equation}
p_{\mathfrak{B}}^\beta (\theta, t) \, = \, \int_0^\infty ds \, p_{\mathfrak{B}} (\theta, s) \, h_\beta (s, t) \, = \, \frac{1}{2\pi} \left[ 1 + 2\sum_{k=1}^\infty e^{- \left( \frac{k^2}{2} \right)^\beta t} \cos k\theta   \right],
\end{equation}
where $p_{\mathfrak{B}}$ is the law of circular Brownian motion and $h_\beta$ is the density of a positively skewed random process of order $\beta \in \left( 0,1 \right]$. In order to check that \eqref{leggebrownspacefrac} solves \eqref{brownspacefrac} it is convenient to write the fractional derivative appearing in \eqref{brownspacefrac} as
\begin{align}
\left( - \frac{1}{2} \frac{\partial^2}{\partial \theta^2} \right)^\beta \, = \, & - \frac{\sin \pi \beta}{\pi} \int_0^\infty  \left( \lambda + \left( -\frac{1}{2} \frac{\partial^2}{\partial \theta^2} \right) \right)^{-1} \lambda^\beta \, d\lambda \notag \\
= \, & - \frac{1}{\Gamma (\beta) \Gamma (1-\beta)} \int_0^\infty  \lambda^\beta \int_0^\infty  e^{-u\lambda - u \left( -\frac{1}{2} \frac{\partial^2}{\partial \theta^2} \right)} \, du \, d\lambda \notag \\
= \, & \frac{1}{\Gamma (-\beta )} \int_0^\infty  u^{-\beta - 1} e^{-u \left( - \frac{1}{2} \frac{\partial^2}{\partial \theta^2} \right)} \, du.
\label{contiformali}
\end{align}
From \eqref{contiformali} we have therefore that
\begin{align}
\left( - \frac{1}{2} \frac{\partial^2}{\partial \theta^2} \right)^\beta \cos k \theta \, = \, & \frac{1}{\Gamma \left( - \beta \right)} \int_0^\infty du \, u^{-\beta -1 } e^{-u \left( - \frac{1}{2} \frac{\partial^2}{\partial \theta^2} \right)} \cos k \theta \notag \\
= \, & \frac{1}{\Gamma (-\beta )} \int_0^\infty du \, u^{-\beta - 1} \sum_{j=0}^\infty \frac{(-u)^j}{j!} \left( - \frac{1}{2} \right)^j \frac{\partial^{2j}}{\partial \theta^{2j}} \cos k \theta \notag \\
= \, & \frac{1 }{\Gamma (-\beta )} \int_0^\infty du \, u^{-\beta - 1} \sum_{j=0}^\infty \frac{(-u)^j}{j!} \frac{k^{2j}}{2^{j}} \cos k \theta \notag \\
= \, & \frac{1}{\Gamma (-\beta )} \int_0^\infty du \, u^{-\beta - 1} e^{-u \frac{k^2}{2}} \cos k \theta \notag \\
 = \, & \left( \frac{k^2}{2} \right)^\beta \cos k \theta,
\end{align}
and this shows that \eqref{leggebrownspacefrac} satisfies \eqref{brownspacefrac}. 
\end{proof}
\begin{os}
Another way to prove that
\begin{equation}
\left( - \frac{1}{2} \frac{\partial^2}{\partial \theta^2} \right)^\beta \cos k \theta \, = \, \left(  \frac{k^2}{2} \right)^\beta \cos k \theta
\end{equation}
can be traced in the paragraph 4.6, page 428 of Balakrishnan \cite{balakrishnan}, which confirms our result.
\end{os}
\begin{te}
For the wrapped up version, say $\mathfrak{S}^{2\beta} (t)$, $t>0$, of the symmetric stable processes $S^{2\beta} (t)$, $t>0$, with characteristic function $\mathbb{E}e^{i\xi S^{2\beta} (t)} = e^{-t | \xi |^{2\beta}}$, we have the following equality in distribution
\begin{align}
\mathfrak{S}^{2\beta} (t) \, \stackrel{\textrm{law}}{=} \,  \mathfrak{B} \left( 2 H^\beta (t) \right) \, \stackrel{\textrm{law}}{=} \, \mathfrak{B} \left( H^\beta  \left( 2^\beta t \right) \right), \qquad t >0.
\end{align}
\end{te}
\begin{proof}
The density of $\mathfrak{S}^{2\beta} (t)$, $t>0$, must be written as
\begin{equation}
p_{\mathfrak{S}^{2\beta}} (\theta, t) \, = \, \frac{\Pr \left\lbrace \mathfrak{S}^{2\beta} (t) \in d\theta \right\rbrace}{ d\theta } \, = \, \frac{1}{2\pi} \sum_{m=-\infty}^\infty \int_{-\infty}^\infty d\xi \, e^{-i\xi (\theta + 2m\pi)} e^{-t | \xi |^{2\beta}}.
\label{stabiliripiegate}
\end{equation}
The Fourier expansion of \eqref{stabiliripiegate} becomes
\begin{equation}
p_{\mathfrak{S}^{2\beta}} (\theta, t) \, = \, \frac{a_0}{2} + \sum_{k=1}^\infty a_k \cos k \theta
\end{equation}
where
\begin{align}
a_k \, = \, & \frac{1}{\pi} \int_0^{2\pi} d\theta \frac{1}{2\pi} \sum_{m=-\infty}^\infty \int_{-\infty}^\infty d\xi \, e^{-i\xi (\theta + 2m\pi )} e^{-t | \xi |^{2\beta}} \cos k \theta \notag \\
= \, & \frac{1}{\pi} \sum_{m=-\infty}^\infty \int_{-\infty}^\infty d\xi \, e^{-t |\xi |^{2\beta}} \int_{m}^{m+1} d\theta \, e^{-i\xi 2\pi \theta} \cos 2 k \pi \theta \notag \\
= \, & \frac{1}{(2\pi)^2} \int_{-\infty}^\infty d\xi \, e^{-t |\xi |^{2\beta}} \int_{-\infty}^\infty dy \, e^{-i\xi y} \left( e^{i y k} + e^{-i y k} \right) \notag \\
= \, & \frac{1}{2\pi} \int_{-\infty}^\infty d\xi \, e^{-t |\xi |^{2\beta}} \left[ \delta (\xi - k) + \delta (\xi + k ) \right] \, = \, \frac{1}{\pi} e^{-tk^{2\beta}}.
\end{align}
This permits us to conclude that
\begin{equation}
p_{\mathfrak{S}^{2\beta}} (\theta, t) \, = \, \frac{1}{2\pi} + \frac{1}{\pi} \sum_{k=1}^\infty e^{-k^{2\beta} t} \cos k \theta.
\label{319}
\end{equation} 
\end{proof}
While the integral in \eqref{stabiliripiegate} (representing the Fourier inverse of symmetric stable laws) cannot carried out, its circular analogue can be explicitely worked out and leads to the Fourier expansion \eqref{319}.
\begin{co}
In view of the results of Theorems \ref{teorematimefrac} and \ref{teoremaspacefrac} we have that the solution to the space-time fractional equation, for $\beta \in \left( 0,1 \right]$,
\begin{align}
\begin{cases}
\frac{\partial^\nu}{\partial t^\nu} p_{\mathfrak{B}}^{\nu, \beta} (\theta, t) \, = \, - \left( - \frac{1}{2} \frac{\partial^2}{\partial \theta^2} \right)^\beta p_{\mathfrak{B}}^{\nu, \beta} (\theta, t), \qquad \theta \in \left[ 0,2\pi \right), \, t>0 \\
p_{\mathfrak{B}}^{\nu, \beta} (\theta, 0) \, = \, \delta (\theta).
\end{cases}
\end{align}
can be written as
\begin{equation}
p_{\mathfrak{B}}^{\nu, \beta} (\theta, t) \, = \, \frac{1}{2\pi} + \frac{1}{\pi} \sum_{k=1}^\infty  E_{\nu, 1} \left( - \left( \frac{k^2}{2} \right)^\beta t^\nu \right) \cos k \theta,
\end{equation}
and coincides with the law of the process
\begin{equation}
\mathpzc{F}^{\nu, \beta} (t) \, = \, \mathfrak{B} \left( H^\beta \left( L^\nu (t) \right) \right), \qquad t>0.
\label{timespaceprocess}
\end{equation}
In \eqref{timespaceprocess} $H^\beta$ is a stable subordinator of order $\beta \in \left( 0,1 \right]$ and $L^\nu$ is the inverse of $H^\nu$ as defined in \eqref{inversosub}.
\end{co}
\begin{proof}
Here we only derive the distribution of $\mathpzc{F}^{\nu, \beta} (t)$, $t>0$. We have that
\begin{align}
\Pr \left\lbrace \mathpzc{F}^{\nu, \beta} (t) \in d\theta \right\rbrace \, = \, & d\theta \int_0^\infty \Pr \left\lbrace \mathfrak{B} (s) \in d\theta \right\rbrace \int_0^\infty \Pr \left\lbrace H^\beta (w) \in ds \right\rbrace \, \Pr \left\lbrace L^\nu (t) \in dw \right\rbrace \notag \\
= \, & \frac{d\theta}{2\pi} + \frac{d\theta}{\pi} \int_0^\infty \sum_{k=1}^\infty e^{-\frac{k^2}{2}s} \cos k \theta \int_0^\infty \Pr \left\lbrace H^\beta (w) \in ds \right\rbrace \, \Pr \left\lbrace L^\nu (t) \in dw \right\rbrace \notag \\
= \, & \frac{d\theta}{2\pi} + \frac{d\theta}{\pi} \int_0^\infty \sum_{k=1}^\infty e^{- \left( \frac{k^2}{2} \right)^\beta w} \cos k \theta \,  \Pr \left\lbrace L^\nu (t) \in dw \right\rbrace \notag \\
= \, & \frac{d\theta}{2\pi} \left[ 1 + 2 \sum_{k=1}^\infty \cos k \theta \, E_{\nu, 1} \left( - \left( \frac{k^2}{2} \right)^\beta t^\nu \right) \right],
\end{align}
where in the last step we applied \eqref{laplinverso}. 
\end{proof}

\section{From pseudoprocesses to Poisson kernels}
In this section we show that the composition of pseudoprocesses of order $n$ running on the circumference $\mathpzc{R}$ with positively skewed stable processes of order $\frac{1}{n}$ leads to the Poisson kernel. This is the circular counterpart of the composition of pseudoprocesses with stable subordinators which leads to Cauchy processes. In both cases pseudoprocesses stopped at $H^{\frac{1}{n}} (t)$, $t>0$, yield genuine random variables.

We distinguish the case where $n$ is even from the case of odd-order pseudoprocesses. We have the first result in Theorem \ref{pseudokernel}.
\begin{te}
\label{pseudokernel}
The composition $\Theta_{2n} \left( H^{\frac{1}{2n}} (t) \right)$, $t>0$, of the pseudoprocess $\Theta_{2n}$ with the stable process $H^{\frac{1}{2n}} (t)$, $t>0$, has density
\begin{align}
\Pr \left\lbrace \Theta_{2n} \left( H^{\frac{1}{2n}} (t)  \right) \in d\theta \right\rbrace \, = \, \frac{d\theta}{2\pi} \frac{1-e^{-2t}}{1+e^{-2t}-2e^{-t}\cos \theta}, \qquad n \in \mathbb{N},
\label{nucleodipoissonteo}
\end{align}
and distribution function
\begin{align}
\Pr \left\lbrace \Theta_{2n} \left( H^{\frac{1}{2n}} (t)  \right) < \theta \right\rbrace \, = \, 
\begin{cases}
\frac{1}{\pi} \arctan \left( \frac{1+e^{-t}}{1-e^{-t}} \tan \frac{\theta}{2} \right), \qquad & \theta \in \left[ 0,\pi \right], \\
1+ \frac{1}{\pi} \arctan \left( \frac{1+e^{-t}}{1-e^{-t}} \tan \frac{\theta}{2} \right), \qquad & \theta \in \left( \pi, 2\pi \right),
\end{cases}
\label{ripartizionenucleo}
\end{align}
which are independent from $n$.
\end{te}
\begin{proof}
We have that
\begin{align}
\Pr \left\lbrace \Theta_{2n} \left( H^{\frac{1}{2n}}  (t) \right) \in d\theta \right\rbrace \, = \, & d\theta \int_0^\infty ds \, \left[ \frac{1}{2\pi} + \frac{1}{\pi} \sum_{k=1}^\infty e^{-k^{2n} s} \cos k \theta \right] \, h_{\frac{1}{2n}} (s,t) \notag \\
= \, & d\theta \left[ \frac{1}{2\pi} + \frac{1}{\pi} \sum_{k=1}^\infty \cos k \theta e^{-kt} \right] \notag \\
= \, & \frac{d\theta}{2\pi} \left[ 1+ \frac{e^{-t+ i \theta}}{1+e^{-t+i\theta}} + \frac{e^{-t-i\theta}}{1-e^{-t-i\theta}} \right] \notag \\
= \, & \frac{d\theta}{2\pi} \frac{1-e^{-2t}}{1+e^{-2t} -2 e^{-t} \cos \theta}.
\end{align}
The result \eqref{ripartizionenucleo} is derived by applying formula 2.552(3) page 172 of Gradshteyn and Ryzhik \cite{G-R}
\begin{align}
\int \frac{dx}{a+b \cos x} \, = \, \frac{2}{\sqrt{a^2-b^2}} \arctan \frac{\sqrt{a^2-b^2} \tan \frac{x}{2}}{a+b}, \qquad a^2 > b^2.
\label{primitivanucleo}
\end{align} 
\end{proof}
\begin{os}
The Poisson kernel \eqref{nucleodipoissonteo} can be interpreted as the distribution of the process $\bm{B} \left( \mathfrak{T}_{\mathpzc{R}} \right)$ where $\bm{B}$ is a planar Brownian motion and $\mathfrak{T}_{\mathpzc{R}} = \inf \left\lbrace t>0 : \bm{B} (t) \in \mathpzc{R} \right\rbrace$. In the case of Theorem \ref{pseudokernel} the planar Brownian motion starts from the point $\left( e^{-t},0 \right)$. Therefore we have that
\begin{equation}
\bm{B} \left( \mathfrak{T}_{\mathpzc{R}} \right) \, \stackrel{\textrm{law}}{=} \, \Theta_{2n} \left( H^{\frac{1}{2n}} (t) \right), \qquad t>0.
\label{leggepseudoprimo}
\end{equation}
This means that a pseudoprocess running on the ring $\mathpzc{R}$ and stopped at a stable time $H^{\frac{1}{2n}} (t)$, $t>0$, has the same distribution of a planar Brownian motion starting from $\left( e^{-t}, 0 \right)$ at the first exit time from the unit-radius circle. The result \eqref{leggepseudoprimo} holds for all $n \in \mathbb{N}$ and represents the circular counterpart of the composition of pseudoprocesses on the line with stable subordinators $H^{\frac{1}{n}} (t)$, $t>0$, which possesses a Cauchy distributed law. As $t \to \infty$ the distribution \eqref{nucleodipoissonteo} converges to the uniform law.
\end{os}

\begin{figure} [htp!]
		\centering
		\caption{In the picture the density of the circular Brownian motion (dotted line) and the kernel \eqref{nucleodipoissonteo} are represented.}
		\includegraphics[scale=0.23]{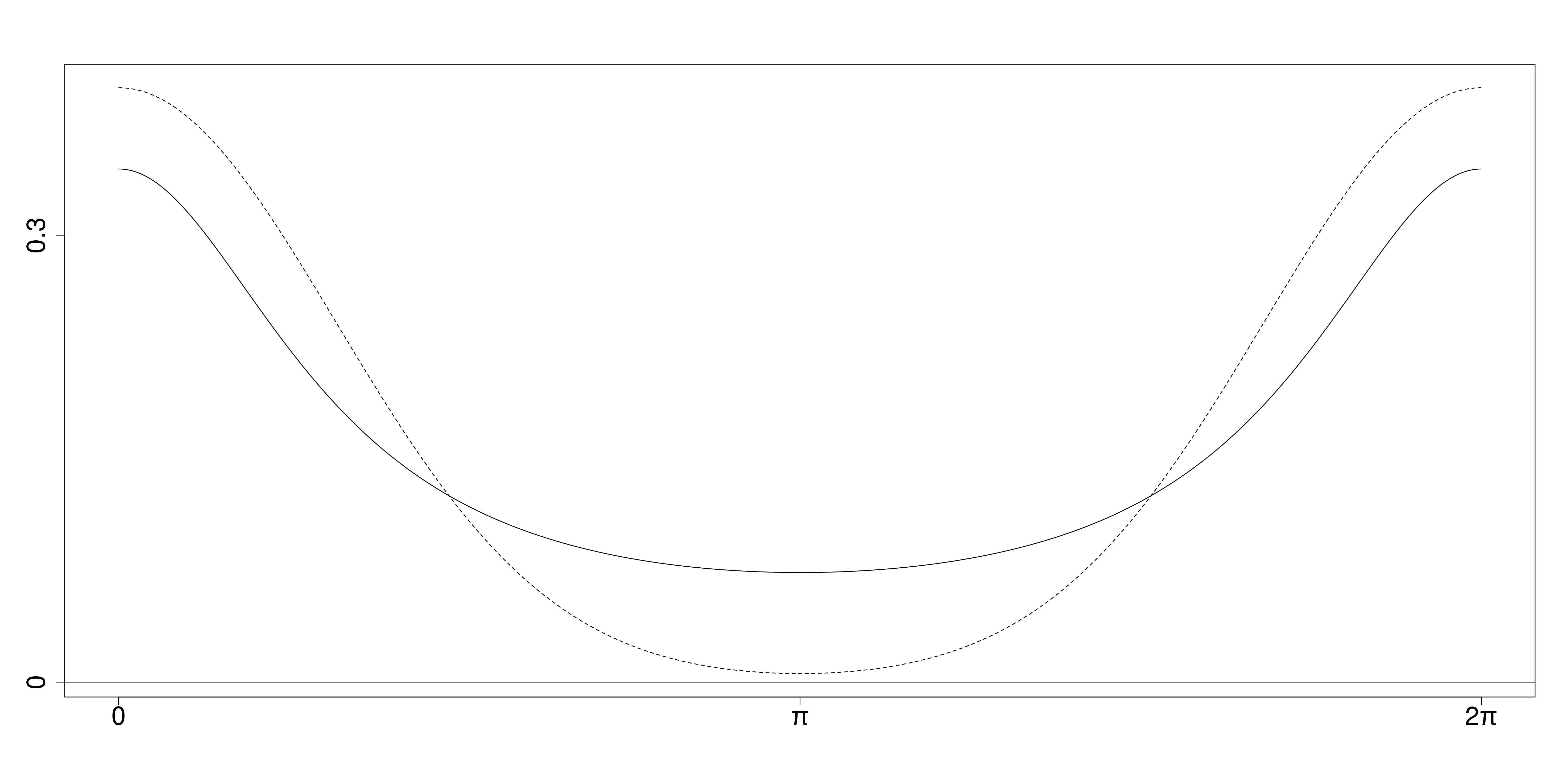} 
		\label{}
\end{figure}

\begin{os}
In view of \eqref{ripartizionenucleo} we note that for $\Theta_{2n} \left( H^{\frac{1}{2n}} (t) \right)$, $t>0$, the probability of staying in the right-hand side of $\mathpzc{R}$ has the remarkably simple form
\begin{equation}
\Pr \left\lbrace -\frac{\pi}{2} <  \Theta_{2n} \left( H^{\frac{1}{2n}} (t) \right) < \frac{\pi}{2} \right\rbrace \, = \, \frac{1}{2} + \frac{2}{\pi} \arctan e^{-t}, \qquad \forall t > 0.
\end{equation}
\end{os}

We now pass to the Poisson kernel associated to odd-order pseudoprocesses. The asymmetry implies that the density of the composition $\Theta_{2n+1} \left( H^{\frac{1}{2n+1}} (t) \right)$, $t>0$, is  bit more complicated than \eqref{nucleodipoissonteo}.
\begin{te}
The composition $\Theta_{2n+1} \left( H^{\frac{1}{2n+1}} (t) \right)$, $t>0$, has density
\begin{equation}
\Pr \left\lbrace \Theta_{2n+1} \left( H^{\frac{1}{2n+1}} (t) \right) \in d\theta \right\rbrace = \frac{d\theta}{2\pi} \frac{1-e^{-2a_nt}}{1+e^{-2a_nt}-2e^{-a_nt}\cos \left( \theta + b_nt \right)},
\label{nucleodispariteo}
\end{equation}
and distribution function
\begin{align}
& \Pr \left\lbrace \Theta_{2n+1} \left( H^{\frac{1}{2n+1}} (t) \right) < \theta \right\rbrace = \notag \\
= \, & 
\begin{cases} \frac{1}{\pi} \left[ \arctan \frac{1+e^{-a_nt}}{1-e^{-a_nt}} \tan \frac{\theta + b_nt}{2} - \arctan \frac{1+e^{-a_nt}}{1-e^{-a_nt}} \tan \frac{b_nt}{2} \right], \quad & 0 < \frac{\theta + b_nt}{2} < \pi,   \\ 1+\frac{1}{\pi} \arctan \frac{1+e^{-a_nt}}{1-e^{-a_nt}} \tan \frac{\theta + b_n t}{2} - \frac{1}{\pi} \arctan \frac{1+e^{-a_n t}}{1-e^{-a_n t}} \tan \frac{b_nt}{2}, & \pi < \theta < 2\pi-\frac{b_nt}{2}.  \end{cases}
\label{ripartiznucleodisp}
\end{align}
where
\begin{equation}
a_n \, = \, \cos \frac{\pi}{2(2n+1)}, \qquad b_n \, = \, \sin \frac{\pi}{2(2n+1)}.
\end{equation}
\end{te}
\begin{proof}
Let $h_{\frac{1}{2n+1}} (s, t)$, $s, t >0$, be the density of a positively skewed stable process of order $\frac{1}{2n+1}$. Then, in view of \eqref{leggipseudoteo}, we have that
\begin{align}
\Pr \left\lbrace \Theta_{2n+1} \left( H^{\frac{1}{2n+1}} (t) \right) \in d\theta \right\rbrace \, = \, &  d\theta \int_0^\infty ds \, \left( \frac{1}{2\pi} + \frac{1}{\pi} \sum_{k=1}^\infty \cos \left( k\theta + k^{2n+1} s \right) \right) \, h_{\frac{1}{2n+1}} (s, t) \notag \\
= \, & d\theta \left[ \frac{1}{2\pi} + \frac{1}{\pi} \sum_{k=1}^\infty e^{-a_n t k} \cos  \left( k (\theta + b_nt)  \right) \right] \notag \\
= \, & \frac{d\theta}{2\pi} \left[ 1+ \frac{e^{i\theta} e^{-t (a_n - i b_n)}}{1-e^{i\theta} e^{-t(a_n - i b_n)}} + \frac{e^{-i\theta} e^{-t (a_n + i b_n)}}{1-e^{-i\theta} e^{-t(a_n + i b_n)}} \right] \notag \\
= \, & \frac{d\theta}{2\pi} \frac{1-e^{-2a_nt}}{1+e^{-2a_nt}-2e^{-a_nt}\cos \left( \theta + b_nt \right)}.
\end{align}
The same result can be obtained by considering that $X_{2n+1} \left( H^{\frac{1}{2n+1}} (t) \right)$ has the following Cauchy distribution (see \cite{ecporsdov})
\begin{equation}
\Pr \left\lbrace X_{2n+1} \left( H^{\frac{1}{2n+1}} (t) \right) \in dx \right\rbrace  \, = \, \frac{t\cos \frac{\pi}{2(2n+1)}}{\pi \left[ \left( x+t \sin \frac{\pi}{2(2n+1)} \right)^2 + t^2 \cos^2 \frac{\pi}{2(2n+1)} \right]}dx.
\label{cauchyasym}
\end{equation}
By wrapping up \eqref{cauchyasym} we arrive at \eqref{nucleodispariteo} in an alternative way.
In view of \eqref{primitivanucleo} we can write
\begin{align}
& \Pr \left\lbrace \Theta_{2n+1} \left( H^{\frac{1}{2n+1}} (t) \right) < \theta \right\rbrace  \, = \notag \\
= \, & \frac{1}{2\pi} \int_0^\theta dy \, \frac{1-e^{-2a_nt}}{1+e^{-2a_nt}-2e^{-a_nt}\cos \left( y + b_nt \right)} \notag \\
= \, & \begin{cases} \frac{1}{\pi} \left[ \arctan \frac{1+e^{-a_nt}}{1-e^{-a_nt}} \tan \frac{\theta + b_nt}{2} - \arctan \frac{1+e^{-a_nt}}{1-e^{-a_nt}} \tan \frac{b_nt}{2} \right], \quad & 0 < \frac{\theta + b_nt}{2} < \pi,   \\ 1+\frac{1}{\pi} \arctan \frac{1+e^{-a_nt}}{1-e^{-a_nt}} \tan \frac{\theta + b_n t}{2} - \frac{1}{\pi} \arctan \frac{1+e^{-a_n t}}{1-e^{-a_n t}} \tan \frac{b_nt}{2}, & \pi < \theta < 2\pi-\frac{b_nt}{2}.  \end{cases}
\label{mammamia}
\end{align} 
\end{proof}

\begin{figure}[htp!]
  \centering
  \subfloat[]{\label{kernparifig}\includegraphics[width=0.5\textwidth]{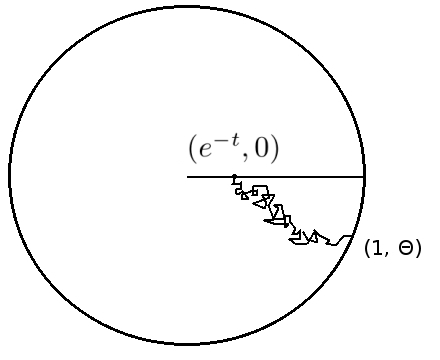}}
  \subfloat[]{\label{kerndispfig}\includegraphics[width=0.5\textwidth]{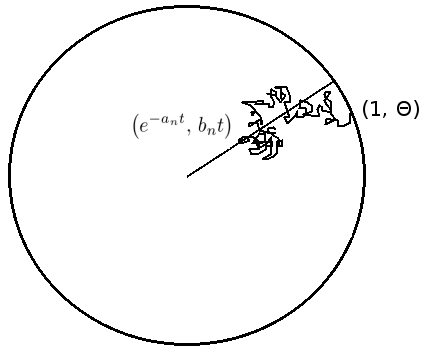}}
  \caption{The distribution of the hitting point of a planar Brownian motion is obtained as subordinated circular pseudoprocess in the even case (Fig. \ref{kernparifig}) and odd case (Fig. \ref{kerndispfig}).}
 \end{figure}
\begin{figure} [htp!]
		\centering
		\includegraphics[scale=0.23]{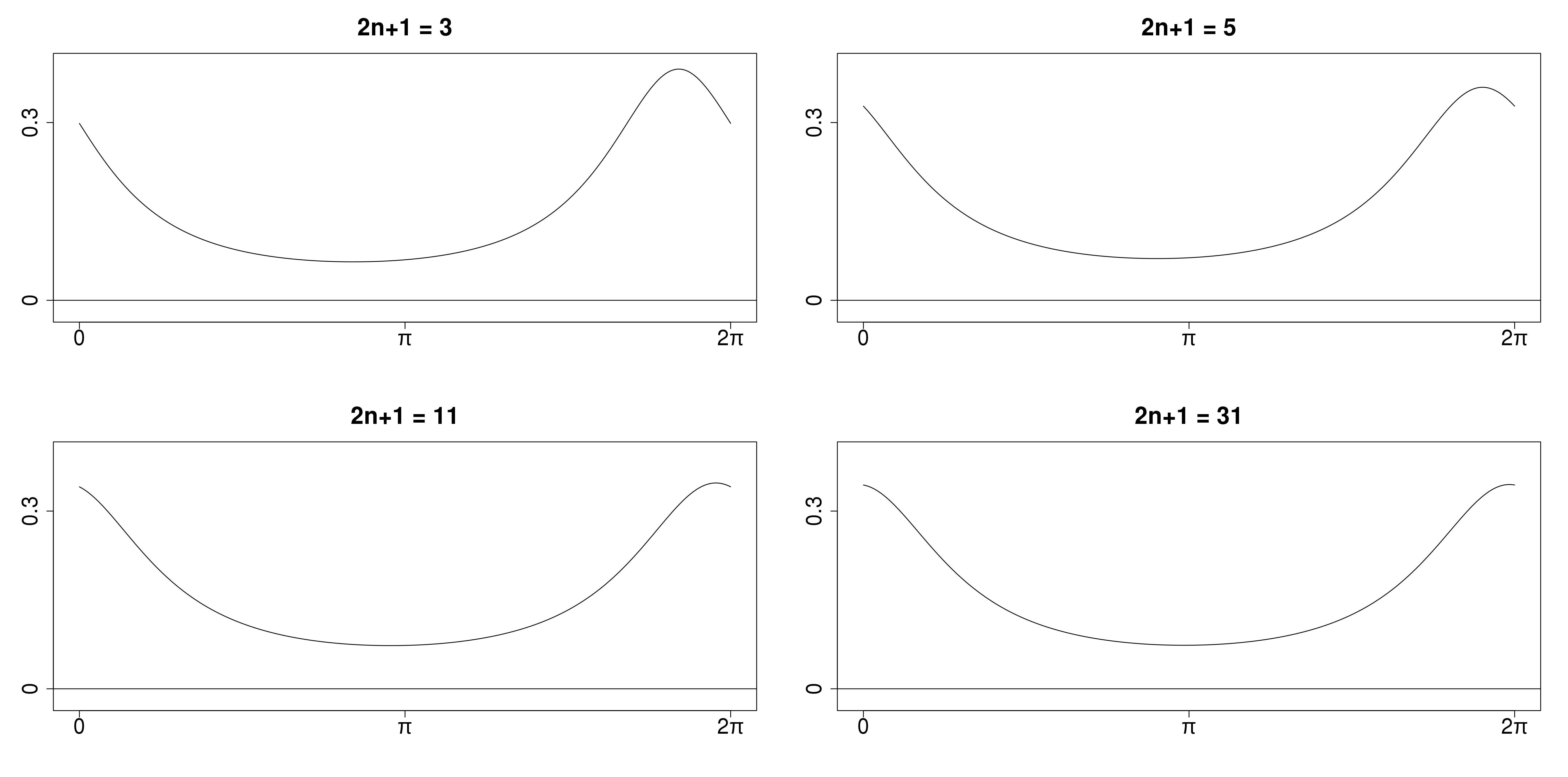} 
		\caption{Distributions related to odd-order Poisson kernels (for $t=1$)}
		\label{}
\end{figure}

\begin{os}
From \eqref{mammamia} we arrive at the following fine expression
\begin{align}
&\Pr \left\lbrace \Theta_{2n+1} \left( H^{\frac{1}{2n+1}} (t) \right) < \theta \right\rbrace \notag \\
 = \, & \frac{1}{\pi} \arctan \left[ \frac{\left( 1-e^{-2a_nt} \right) \tan \frac{\theta}{2} \left( 1+\tan^2 \frac{b_nt}{2} \right)}{ \left( 1-e^{-a_nt} \right)^2 + 4 \tan \frac{\theta}{2} \tan \frac{b_nt}{2} + \left( 1+e^{-a_nt} \right)^2 \tan^2 \frac{b_nt}{2} } \right], \qquad \theta \in [0,2\pi),
\end{align}
from which we are able to explicitely write for $\Theta_{2n+1} \left( H^{\frac{1}{2n+1}} (t) \right)$ the probability of staying in the interval $(0, \pi)$ as
\begin{equation}
\Pr \left\lbrace 0 < \Theta_{2n+1} \left( H^{\frac{1}{2n+1}} (t) \right) < \pi \right\rbrace \, = \, \frac{1}{\pi} \arctan \frac{\sinh a_nt}{\sin b_n t}, \qquad \forall t >0,
\end{equation}
while for $\left( 0,\frac{\pi}{2} \right)$ we obtain
\begin{align}
& \Pr \left\lbrace 0< \Theta_{2n+1} \left( H^{\frac{1}{2n+1}} (t) \right) < \frac{\pi}{2} \right\rbrace \notag \\
 = \, & \frac{1}{\pi} \arctan \frac{\left( 1-e^{-2a_nt} \right) \left( 1+ \tan^2 \frac{b_nt}{2} \right)}{ \left( 1-e^{-a_nt} \right)^2 + 4 \tan \frac{b_nt}{2} + \left( 1+e^{-a_nt} \right)^2 \tan^2 \frac{b_nt}{2} } \notag \\
= \, & \frac{1}{\pi} \arctan \frac{ \sinh a_n t}{2 \sinh^2 \frac{a_nt}{2} \cos^2 \frac{b_nt}{2} + e^{a_nt} \sin b_nt + 2 \cosh^2 \frac{a_nt}{2}\sin^2 \frac{b_nt}{2}} \notag \\
= \, & \frac{1}{\pi} \arctan \frac{\sinh a_n t}{\cosh a_n t - \cos b_n t + e^{a_nt} \sin b_n t}.
\label{chepalle}
\end{align}
By means of the same manipulations leading to \eqref{chepalle} we arrive at the alternative form of the distribution function for $\theta \in [0,\pi]$,
\begin{align}
\Pr \left\lbrace 0 < \Theta_{2n+1} \left( H^{\frac{1}{2n+1}} (t) \right)  < \theta \right\rbrace \, = \, \frac{1}{\pi} \arctan \frac{\sinh a_n t \, \tan \frac{\theta}{2}}{\cosh a_n t - \cos b_n t + e^{a_nt} \sin b_n t \, \tan \frac{\theta}{2}}.
\end{align}
\end{os}

\begin{os}
In the third-order case we can arrive at the Poisson kernel \eqref{nucleodispariteo} for $n=1$ by considering that (see \cite{ecporsdov})
\begin{align}
\frac{\Pr \left\lbrace X_3 \left( H^{\frac{1}{3}} (t) \right) \in dx \right\rbrace}{ dx } \, = \, & \int_0^\infty \frac{ds}{\sqrt[3]{3s}} \, \textrm{Ai} \left( \frac{x}{\sqrt[3]{3s}} \right) \, \frac{t}{s\sqrt[3]{3s}} \, \textrm{Ai} \left( \frac{t}{\sqrt[3]{3s}} \right) \notag \\
 = \, & \frac{\sqrt{3}t}{2 \pi \left[ \left( x+ \frac{t}{2} \right)^2 + \frac{3t^2}{4} \right]}.
\label{integrairy}
\end{align}
The wrapped up counterpart of \eqref{integrairy} becomes for $\theta \in \left[ 0,2\pi \right)$
\begin{align}
& \Pr \left\lbrace \Theta_3 \left( H^{\frac{1}{3}} (t) \right) \in d\theta \right\rbrace \, = \,  \sum_{m=-\infty}^\infty  \Pr \left\lbrace X_3 \left( H^{\frac{1}{3}} (t) \right)  \in d(\theta + 2m\pi) \right\rbrace \, = \notag \\
= \, &  d\theta \sum_{m=-\infty}^\infty \int_0^\infty \frac{ds}{\sqrt[3]{3s}}  \textrm{Ai} \left( \frac{\theta + 2m\pi}{\sqrt[3]{3s}} \right)  \frac{t}{s\sqrt[3]{3s}}  \textrm{Ai} \left( \frac{t}{\sqrt[3]{3s}} \right) \notag \\
= \, & d\theta \frac{\sqrt{3}t}{2\pi} \sum_{m=-\infty}^\infty \frac{t}{\left( \theta + 2m\pi + \frac{t}{2} \right)^2 + 3\frac{t^2}{4}} \, = \,  \frac{d\theta \sqrt{3} t}{2} \sum_{m=-\infty}^\infty \int_0^\infty ds \frac{e^{-\frac{\left( \theta + 2m\pi + \frac{t}{2} \right)^2}{2s}}}{\sqrt{2\pi s}} \frac{e^{-\frac{\left( \sqrt{3}t /2 \right)^2}{2s}}}{\sqrt{2\pi s^3}} \notag \\
= \, & \frac{d\theta \sqrt{3}t}{2^2\pi} \int_0^\infty ds \left[ 1+2\sum_{k=1}^\infty \cos \left[ k \left( \theta + \frac{t}{2} \right) \right] e^{-\frac{k^2s}{2}} \right] \frac{e^{-\frac{\left( \sqrt{3}t /2 \right)^2}{2s}}}{\sqrt{2\pi s^3}}  \notag \\
 = \,  & \frac{d\theta}{2\pi} \left[ 1+2\sum_{k=1}^\infty e^{-\frac{\sqrt{3}}{2}tk} \cos \left[ k \left( \theta + \frac{t}{2} \right) \right] \right] \, = \,  \frac{d\theta}{2\pi} \frac{1-e^{-\sqrt{3}t}}{1+e^{-\sqrt{3}t}-2e^{-\sqrt{3}t} \cos \left( \theta + \frac{t}{2} \right)},
\end{align}
which coincides with \eqref{nucleodispariteo} since $a_1 = \frac{\sqrt{3}}{2}$ and $b_1 = \frac{1}{2}$.
\end{os}

\end{document}